\documentclass{amsart}
\usepackage{graphicx}
\newtheorem{thm}{Theorem}[section]
\newtheorem{cor}[thm]{Corollary}
\newtheorem{lem}[thm]{Lemma}

\theoremstyle{definition}

\theoremstyle{example}
\newtheorem{exam}{Example}
\theoremstyle{remark}
\newtheorem{rem}[thm]{Remark}
\numberwithin{equation}{section}
\begin{document}

\title[Binomial sums]{Parametric  binomial sums involving harmonic numbers}%
\author{necdet bat{\i}r}%
\address{department of  mathematics, nev{\c{s}}ehir hbv university, nev{\c{s}}ehir, 50300 turkey}%
\email{nbatir@hotmail.com}%

%\thanks{}%

%\dedicatory{}%
%\commby{}%
% ----------------------------------------------------------------
\begin{abstract} We present explicit formulas for the following family of parametric binomial sums involving harmonic numbers for $p=0,1,2$ and $|t|\leq1$.
$$
\sum_{k=1}^{\infty}\frac{H_{k-1}t^k}{k^p\binom{n+k}{k}}\quad \mbox{and}\quad \sum_{k=1}^{\infty}\frac{t^k}{k^p\binom{n+k}{k}}.
$$
We also generalize the following  relation between the Stirling numbers of the first kind and the Riemann zeta function to polygamma function and give some applications.
$$
\zeta(n+1)=\sum_{k=n}^{\infty}\frac{s(k,n)}{kk!}, \quad n=1,2,3,... .
$$
As examples,
\begin{equation*}
\zeta(3)=\frac{1}{7}\sum_{k=1}^{\infty}\frac{H_{k-1}4^k}{k^2\binom{2k}{k}},\quad \mbox{and}\quad \zeta(3)=\frac{8}{7}+\frac{1}{7}\sum_{k=1}^{\infty}\frac{H_{k-1}4^k}{k^2(2k+1)\binom{2k}{k}},
\end{equation*}
which are new series representations for the Ap\'{e}ry constant $\zeta(3)$.
\end{abstract}
\maketitle
% ----------------------------------------------------------------
\section{introduction}
First, we recall some tools we need. The gamma and digamma  functions denoted by  $\Gamma$  and $\psi$, respectively, are defined by
$$
\Gamma(x)=\int_{0}^\infty u^{x-1}e^{-u}du\quad \mbox{and}\quad \psi(x)=\frac{\Gamma'(x)}{\Gamma(x)},\quad (x>0).
$$
The gamma function satisfies the Legendre's duplication formula
\begin{equation}\label{e:1}
\Gamma(n+1/2)=\frac{(2n)!\sqrt{\pi}}{4^nn!},\quad n=0,1,2,3,... .
\end{equation}
The functions $\psi,\psi',\psi'',...$ are known to be the polygamma functions.
$\psi^{(0)}(x)=\psi(x)$.
The polygamma functions satisfy the following fundamental functional equation:
\begin{equation}\label{e:2}
\psi^{(n)}(x+1)-\psi^{(n)}(x)=\frac{(-1)^nn!}{x^{n+1}},\quad n=0,1,2,... ,
\end{equation}
and have the integral representation
\begin{equation}\label{e:3}
\psi^{(n)}(s+1)=(-1)^{n+1}\int_{0}^{\infty}\frac{t^ne^{-(s+1)t}}{1-e^{-t}}dt,
\end{equation}
see \cite{16}. The polygamma functions satisfy
\begin{equation}\label{e:4}
\psi^{(n)}(1)=(-1)^{n-1}n!\zeta(n+1) \quad \mbox{and}\quad \psi^{(n)}(1/2)=(-1)^{n+1}n!\left(2^{n+1}-1\right)\zeta(n+1),
\end{equation}
where $\zeta(s)=\sum_{k=1}^{\infty}\frac{1}{k^s}$\quad $(\Re(s)>1)$ is the Riemann zeta function. The beta function $B(s,t)$  is defined by
$$
B(s,t)=\int\limits_{0}^{1}u^{s-1}(1-u)^{t-1}du\quad (\Re(s)>0\,,\Re(t)>0).
$$
The gamma and beta functions are  related by
\begin{equation}\label{e:5} %=================================== Equation (1.11)
B(s,t)=\frac{\Gamma(s)\Gamma(t)}{\Gamma(s+t)};
\end{equation}
see \cite[p.251]{11}.
A generalized binomial coefficient  $\binom{s}{t}$ ($s,t\in\mathbb{C}$) is defined, in terms of the classical gamma function, by
\begin{equation}\label{e:6}
\binom{s}{t}=\frac{\Gamma(s+1)}{\Gamma(t+1)\Gamma(s-t+1)},\quad(s,t\in\mathbb{C}).
\end{equation}
For $s\in\mathbb{C}$, a generalized harmonic number $H_n^{(s)}$ of order $s$ is defined by
\begin{equation}\label{e:7} %=================================== Equation (1.1)
H_n^{(s)}=\sum\limits_{k=1}^{n}\frac{1}{k^s}, \quad \mbox{and} \quad H_n^{(1)}=H_n;
\end{equation}
see \cite{28}. Here and throughout, an empty sum is understood to be nil and so $H_0{(s)}=0$. We also use the principal branch of  logarithmic function $\log z$ for $z\in\mathbb{C}$. Recall the relation between harmonic numbers and the digamma function: When $n\in\mathbb{N}$, $\gamma+\psi(n+1)=H_n$.
The generalized harmonic numbers and the polygamma functions are related with
\begin{equation}\label{e:8}
H_n^{(m)}=\zeta(m+1)+\frac{(-1)^m}{m!}\psi^{(m)}(n+1),\quad n\in \mathbb{N};
\end{equation}
see \cite{25}.
Here and throughout, an empty sum is understood to be nil and so $H_0{(s)}=0$.
The Stirling numbers of the first kind are defined by the generating function of the falling factorial $(x)_n=x(x-1)(x-2)...(x-(n-1))$ $(n\geq 0)$, that is,
$$
\sum_{k=0}^{\infty}s(n,k)x^k=(x)_n.
$$
Some special values include
\begin{align}\label{e:9}
s(n,0)&=\delta_{n0},\quad s(n,1)=(-1)^{n-1}(n-1)!,\quad s(n,2)=(-1)^{n}(n-1)!H_{n-1},\nonumber\\
s(n,3)&=\frac{1}{2}(-1)^{n-1}(n-1)!\left(H_{n-1}^2-H_{n-1}^{(2)}\right),\notag\\
s(n,4)&=\frac{1}{6}(-1)^n(n-1)!\left(H_{n-1}^3-3H_{n-1}H_{n-1}^{(2)}+2H_{n-1}^{(3)}\right).
\end{align}
Here $\delta_{ij}$ is the Kronecker delta. For basic properties, closed form evaluations, and generalizations  for many other special numbers and polynomials, such as, Stirling numbers of the second kind, Bernoulli and Euler polynomials, Eulerian and Catalan numbers, Daehee numbers we refer the interested readers to Comtet \cite{9}, Da\v{g}l{\i} \cite{Dag}, Gun and \c{S}im\c{s}ek \cite{Gun}, and Kim and Kim \cite{Kim}, and the papers cited there.

The polylogarithmic function $Li_n(x)$ is defined by
$$
Li_n(x)=\sum_{k=1}^{\infty}\frac{x^k}{k^n},
$$
where here  requires $-1\leq x<1$ if $n=1$ and $|x|\leq 1$ if $n>1$. In particular, the dilogarithm function $Li_2$ satisfies the  functional equation
\begin{equation}\label{e:10}
Li_2(x)+Li_2(1-x)=\zeta(2)-\log x\log(1-x);
\end{equation}
see \cite{15} for other basic properties of polylogarithm functions. It is well known that binomial coefficients play an important role in various subjects such as combinatorics, number theory, probability, and graph theory. In the literature there exist many results dealing with sums involving reciprocals of the binomial coefficients. Particularly, sums involving both reciprocals of the binomial coefficients and harmonic numbers  have been receiving much attention. For example, in \cite{9} the author gives closed-form identities for the sums
$$
\sum_{k=1}^{\infty}\frac{(-1)^{k+1}H_k}{k^p\binom{n+k}{k}^2},\quad p=0,1.
$$
In \cite{20} the author gives explicit formulas for the following sums:
$$
\sum_{k=1}^{\infty}\frac{(-1)^{k+1}H_k^2}{k^p\binom{n+k}{k}^2},\quad p=0,1.
$$
In \cite{21} the author gives explicit identities for the following sum for $t=1$ and $1/2$:
$$
\sum_{k=1}^{\infty}\frac{t^{k}H_k}{\binom{n+k}{k}}.
$$
Cloitre as reported in \cite{29} gives the formula
$$
\sum_{k=1}^{\infty}\frac{H_k}{\binom{n+k}{k}}=\frac{n}{(n-1)^2}.
$$
In \cite{18} the author gives some complicated formulas involving generalized harmonic numbers and binomial coefficients for
$$
\sum_{k=1}^{\infty}\frac{(-1)^{k+1}H_k}{k^p\binom{n+k}{k}},\quad p=0,1,2.
$$
In \cite{30} the authors give elegant identities for the following sums:
$$
\sum_{k=1}^{\infty}\frac{(-1)^k}{\binom{n+k}{k}}\quad \mbox{and}\quad \sum_{k=1}^{\infty}\frac{1}{2^k\binom{n+k}{k}}
$$
For many other infinite series involving the harmonic numbers and binomial coefficients please refer to  [3], [4], [22]-[28], [30]-[32], [33]-[35] and the references cited there.
Motivated by these results mentioned above in this paper we consider the following family of sums:
$$
\sum_{k=1}^{\infty}\frac{t^k}{k^p\binom{n+k}{k}} \quad\mbox{and} \quad \sum_{k=1}^{\infty}\frac{H_{k-1}t^k}{k^p\binom{n+k}{k}},\quad p=0,1,2,
$$
for $|t|<1$. While there are many results for sums of the types given above there are almost no results for the power series like these. It is a well-known fact that evaluating  the sum of this kind of   series explicitly is not easy, and usually it is done by using the gamma-beta functions identity or a combinatorial formula for the reciprocal of the pochhammer symbol. This formula has been used extensively by A. Sofo in his works. But we believe that using proper integrals are more useful for studying these sums and therefore we employ some logarithmic integrals and by this way we provide elegant explicit formulas for these sums.

In Jordan's book \cite[pp. 166, 194-195]{14} the following interesting identity, which establishes an interesting relation between $\zeta(n)$ and $s(n,k)$, was recorded:
\begin{equation}\label{e:11}
\zeta(n+1)=\sum_{k=n}^{\infty}\frac{s(k,n)}{kk!}, \quad n=1,2,3,... .
\end{equation}
This formula was rediscovered by some authors recently; see, \cite{17}, \cite{16}, \cite[Sec. 5]{1} and \cite[p. 412]{7}. Our second aim in this work  is to generalize this formula to polygamma functions and give some applications.

We need the following useful lemmas. The first lemma is proved in \cite{5}.
\begin{lem} Let $(a_k)_{k\geq1}$ and $(b_k)_{k\geq1}$ be any two sequences of real or complex numbers. Then we have
\begin{equation}\label{e:12}
\sum\limits_{k=1}^{n}a_k\sum\limits_{j=1}^{k}b_j=\sum\limits_{p=0}^{n-1}\sum\limits_{k=1}^{n-p}b_ka_{p+k}.
\end{equation}
\end{lem}
If we let $n\to \infty$ here, we get
\begin{equation}\label{e:13}
\sum\limits_{k=1}^{\infty}a_k\sum\limits_{j=1}^{k}b_j=\sum\limits_{p=0}^{\infty}\sum\limits_{k=1}^{\infty}b_ka_{p+k},
\end{equation}
provided that all involved series are convergent.
\begin{lem} For $n\in\mathbb{N}$ we have
\begin{equation*}
\sum_{k=1}^{n}\frac{(-a)^k}{k}\sum_{J=1}^{k}\frac{(-1/a)^j}{j}=H_n^{(2)}+\sum_{k=1}^{n}\frac{(H_k+H_{n-k}-H_n)(-a)^k}{k}.
\end{equation*}
\end{lem}
\begin{proof}
The proof follows from (\ref{e:12}) with $a_k=(-a)^k/k$ and $b_j=(-1/a)^j/j$.
\end{proof}
The following lemma is due to Hassani and Rahimpour [13].
\begin{lem} Let $(a_{ij})$ be any double sequence. Then we have
\begin{equation*}
\sum_{j,k=1}^{n}a_{jk}=\sum_{k=1}^n\sum_{j=1}^{k}(a_{jk}+a_{kj})-\sum_{k=1}^{n}a_{kk}.
\end{equation*}
\end{lem}
\begin{lem} Let $n\in\mathbb{N}$  and $a\in\mathbb{C}$. Then, we have
\begin{align}\label{e:14}
&\sum_{k=1}^{n}\frac{(-a)^k}{k}\sum_{j=1}^{k}\frac{1}{j}\sum_{p=1}^{j}\frac{(-1/a)^p}{p}\notag\\
&=\sum_{p=0}^{n-1}\bigg\{\sum_{k=1}^{n-p}\bigg(\frac{H_{p+k}-H_k}{(p+k)k}+\frac{1}{k^2(p+k)}\bigg)\bigg\}(-a)^p.
\end{align}
\end{lem}
\begin{proof}
From Lemma 1.3 with $a_{jk}=\frac{(-1/a)^j}{jk}$ we get
\begin{align}\label{e:15}
\sum_{j=1}^{k}\frac{1}{j}\sum_{p=1}^{j}\frac{(-1/a)^p}{p}&=\sum_{j=1}^{k}\frac{H_k(-1/a)^j}{j}-\sum_{j=1}^{k}\frac{H_j(-1/a)^j}{j}\notag\\
&+\sum_{j=1}^k\frac{(-1/a)^j}{j^2}.
\end{align}
Multiplying both sides of (\ref{e:15}) by $\frac{(-a)^k}{k}$ and then summing both sides of the resulting equation from $k=1$ to $k=n$, it follows that
\begin{align}\label{e:16}
&\sum_{k=1}^{n}\frac{(-a)^k}{k}\sum_{j=1}^{k}\frac{1}{j}\sum_{p=1}^{j}\frac{(-1/a)^p}{p}=\sum_{k=1}^{n}\frac{H_k(-a)^k}{k}\sum_{j=1}^{k}\frac{(-1/a)^j}{j}\notag\\
&-\sum_{k=1}^{n}\frac{(-a)^k}{k}\sum_{j=1}^{k}\frac{H_j(-1/a)^j}{j}+\sum_{k=1}^{n}\frac{(-a)^k}{k}\sum_{j=1}^k\frac{(-1/a)^j}{j^2}.
\end{align}
Choosing the sequences $a_n$ and $b_n$ in (\ref{e:12}) appropriately  we obtain the following identities:
\begin{equation}\label{e:17}
\sum_{k=1}^{n}\frac{H_k(-a)^k}{k}\sum_{j=1}^{k}\frac{(-1/a)^j}{j}=\sum_{p=0}^{n-1}\bigg(\sum_{k=1}^{n-p}\frac{H_{p+k}}{k(p+k)}\bigg)(-a)^p,
\end{equation}
\begin{equation}\label{e:18}
\sum_{k=1}^{n}\frac{(-a)^k}{k}\sum_{j=1}^{k}\frac{H_j(-1/a)^j}{j}=\sum_{p=0}^{n-1}\bigg(\sum_{k=1}^{n-p}\frac{H_{k}}{k(p+k)}\bigg)(-a)^p,
\end{equation}
and
\begin{equation}\label{e:19}
\sum_{k=1}^{n}\frac{(-a)^k}{k}\sum_{j=1}^k\frac{(-1/a)^j}{j^2}=\sum_{p=0}^{n-1}\bigg(\sum_{k=1}^{n-p}\frac{1}{k^2(p+k)}\bigg)(-a)^p,
\end{equation}
Substituting (\ref{e:17}), (\ref{e:18}) and (\ref{e:19}) in (\ref{e:16}) we get the desired conclusion.
\end{proof}
\begin{rem}
We want to remark that the right hand side of (\ref{e:14}) is a polynomial of degree $n-1$ in $a$.
\end{rem}
\begin{lem}Let $a\in\mathbb{R}\backslash (-2,0)$, and $s$ be a positive real number. Then the following equality holds:
\begin{align*}
&\int_{0}^{1}\frac{1-x^s}{1-x}\log\left(1-\frac{1-x}{1+a}\right)dx=\sum_{k=1}^{\infty}\frac{1}{k^2(a+1)^k\binom{s+k}{k}}-Li_2\left(\frac{1}{a+1}\right).
\end{align*}
\end{lem}
\begin{proof} Expanding the logarithmic function to its power series, we get
\begin{align}\label{uniform}
\int_{0}^{1}\frac{1-x^s}{1-x}\log\left(1-\frac{1-x}{1+a}\right)dx&=-\int_{0}^{1}(1-x^s)\sum_{k=1}^{\infty}\frac{(1-x)^{k-1}}{k(a+1)^k}dx\\
&=-\int_{0}^{1}(1-(1-x)x^s)\sum_{k=1}^{\infty}\frac{x^{k-1}}{k(a+1)^k}dx.
\end{align}
Let $x\in [0,1]$, and $a\in\mathbb{R}\backslash [-2,0]$.  Then we, clearly, have
$$
\left|\frac{(1-(1-x)^s)x^{k-1}}{k(a+1)^k}\right|\leq \frac{1}{k|a+1|^k},
$$
and $\sum_{k=1}^{\infty}\frac{1}{k[a+1]^k}<\infty$, so the series $\sum_{k=1}^{\infty}\frac{(1-(1-x)^s)x^{k-1}}{k(a+1)^k}$ is uniformly convergent on [0,1] by Weierstrass $M$-test. Therefore, we can reverse the order of integration and summation in (\ref{uniform}) (see \cite[pp.192--193]{Ter}), and
\begin{align*}
&\int_{0}^{1}\frac{1-x^s}{1-x}\log\left(1-\frac{1-x}{1+a}\right)dx\\
&=-\sum_{k=1}^{\infty}\frac{1}{k(a+1)^k}\left[\int_{0}^{1}(1-x)^{k-1}dx-\int_{0}^{1}x^s(1-x)^{k-1}dx\right].
\end{align*}
Since  we have by (\ref{e:5}) and (\ref{e:6})
\begin{equation}\label{e:20}
\int_{0}^{1}(1-x)^{k-1}dx=\frac{1}{k} \quad \mbox{and} \quad \int_{0}^{1}x^s(1-x)^{k-1}dx=\frac{1}{k\binom{s+k}{k}},
\end{equation}
we find for $a\in\mathbb{R}\backslash [-2,0]$
\begin{align*}
\int_{0}^{1}\frac{1-x^s}{1-x}\log\left(1-\frac{1-x}{1+a}\right)dx=\sum_{k=1}^{\infty}\frac{1}{k^2(a+1)^k\binom{s+k}{k}}-Li_2\left(\frac{1}{a+1}\right).
\end{align*}

But this series is also convergent at $a=-2$ and $a=0$, so, Lemma 1.6 is valid for $a\in\mathbb{R}\backslash (-2,0)$ by Abel's limit test. This completes the proof.
\end{proof}
\begin{lem} For k=0,1,2,...,  and $a\in\mathbb{R}$ we have
\begin{align*}
\int_{0}^{1}x^k\log (x+a)dx&=\frac{\log(a+1)}{k+1}-\frac{(-a)^{k+1}}{k+1}\sum_{j=1}^{k+1}\frac{(-1/a)^j}{j}\\
&-\frac{(-a)^{k+1}}{k+1}\log(1+1/a).
\end{align*}
\end{lem}
\begin{proof}
Integrating by parts yields
\begin{equation}\label{e:21}
\int_{0}^{1}x^k\log(x+a)dx=\frac{\log(a+1)}{k+1}-\frac{1}{k+1}\int_{0}^{1}\frac{x^{k+1}}{x+a}dx.
\end{equation}
On the other hand, we have
\begin{align*}
&\int_{0}^{1}\frac{x^{k+1}}{x+a}dx=(-1)^{k}a^k\int_{0}^{1}\frac{1-(-x/a)^{k+1}-1}{1-(-x/a)}dx\nonumber\\
&=(-a)^k\int_{0}^{1}\frac{1-(-x/a)^{k+1}}{1-(-x/a)}dx+(-a)^{k+1}\int_{0}^{1}\frac{dx}{x+a}\nonumber\\
&=(-a)^{k}\int_{0}^{1}\sum_{j=0}^{k}(-x/a)^jdx+(-a)^{k+1}\log\left(1+\frac{1}{a}\right).
\end{align*}
Inverting the order of summation and integration, we obtain, after a short computation,
\begin{equation*}
\int_{0}^{1}\frac{x^{k+1}}{x+a}dx=(-a)^{k+1}\sum_{j=1}^{k+1}\frac{(-a)^j}{j}+(-a)^{k+1}\log\left(1+\frac{1}{a}\right).
\end{equation*}
Substituting this into (\ref{e:21}), we complete the proof.
\end{proof}
\begin{lem} For $|x|\leq1$ and $k\in\mathbb{N}$ we have
\begin{align*}
&\sum_{j=1}^{\infty}\frac{H_jx^j}{(j+1)(j+k+1)}=\frac{1}{2xk}\left(1-\frac{1}{x^k}\right)\log^2(1-x)\nonumber\\
&+\frac{\log(1-x)}{kx^{k+1}}\left(H_k-\sum_{j=1}^{k}\frac{x^j}{j}\right)+\frac{1}{kx^{k+1}}\sum_{j=1}^{k}\frac{1}{j}\sum_{p=1}^{j}\frac{x^p}{p}.
\end{align*}
\end{lem}
\begin{proof} By partial fraction decomposition we get
\begin{align}\label{e:22}
&\sum_{j=1}^{\infty}\frac{H_jx^j}{(j+1)(j+k+1)}=\frac{1}{k}\sum_{j=1}^{\infty}H_j\left(\frac{1}{j+1}-\frac{1}{j+k+1}\right)x^j\nonumber\\
&=\frac{1}{k}\sum_{j=1}^{\infty}\frac{H_j x^j}{j+1}-\frac{1}{k}\sum_{j=1}^{\infty}\frac{H_j x^j}{j+k+1}.
\end{align}
Employing (\ref{e:13}) with $b_j=1/j$ and $a_j=\frac{x^j}{j+k+1}$, we get
\begin{align*}
&\sum_{j=1}^{\infty}\frac{H_jx^j}{j+k+1}=\sum_{p=0}^{\infty}\sum_{j=1}^{\infty}\frac{x^{p+j}}{j(p+j+k+1)}.
\end{align*}
From \cite{10} we know that
\begin{equation}\label{e:23}
\sum_{j=1}^{\infty}\frac{H_jx^j}{j}=\frac{1}{2}\log^2(1-x)+Li_2(x).
\end{equation}
By partial fraction decomposition, we obtain
\begin{align*}
&\sum_{j=1}^{\infty}\frac{H_jx^j}{j+k+1}=\sum_{p=0}^{\infty}\sum_{j=1}^{\infty}\frac{x^{p+j}}{p+k+1}\left(\frac{1}{j}-\frac{1}{p+j+k+1}\right)\nonumber\\
&=\sum_{p=0}^{\infty}\frac{x^p}{p+k+1}\left(\sum_{j=1}^{\infty}\frac{x^j}{j}-\sum_{j=1}^{\infty}\frac{x^j}{p+j+k+1}\right)\nonumber\\
&=-\log(1-x)\sum_{p=0}^{\infty}\frac{x^p}{p+k+1}-\sum_{p=0}^{\infty}\frac{x^p}{p+k+1}\sum_{j=1}^{\infty}\frac{x^j}{p+j+k+1}.
\end{align*}
Setting $p+k+1=k^\prime$ and then deleting the prime one gets
\begin{align}\label{A}
&\sum_{j=1}^{\infty}\frac{H_jx^j}{j+k+1}=-\frac{\log(1-x)}{x^{k+1}}\sum_{p=k+1}^{\infty}\frac{x^p}{p}-\frac{1}{x^{k+1}}\sum_{p=k+1}^{\infty}\frac{x^p}{p}\sum_{j=1}^{\infty}\frac{x^j}{p+j}\notag\\
&=-\frac{\log(1-x)}{x^{k+1}}\left(-\log(1-x)-\sum_{p=1}^{k}\frac{x^p}{p}\right)\notag\\
&-\frac{1}{x^{k+1}}\sum_{p=1}^{\infty}\sum_{j=1}^{\infty}\frac{x^{p+j}}{p(p+j)}+\frac{1}{x^{k+1}}\sum_{p=1}^{k}\sum_{j=1}^{\infty}\frac{x^{p+j}}{p(p+j)}.
\end{align}
Since $|x|<1$, we have
\begin{align*}
\sum_{p=1}^{\infty}\sum_{j=1}^{\infty}\left|\frac{x^{p+j}}{p(p+j)}\right|<\sum_{p=1}^{\infty}\sum_{j=1}^{\infty}\left|\frac{x^{p+j}}{p(j+1)}\right|=\sum_{p=1}^{\infty}\frac{|x|^p}{p}\sum_{j=1}^{\infty}\frac{|x|^j}{j+1}<\infty.
\end{align*}
We therefore can reverse the order of sums in the first sum in the last line of (\ref{A}) and obtain by making use of (\ref{e:13})
\begin{align*}
\sum_{p=1}^{\infty}\sum_{j=1}^{\infty}\frac{x^{p+j}}{p(p+j)}&=\sum_{j=1}^{\infty}\sum_{p=1}^{\infty}\frac{x^{p+j}}{p(p+j)}=\sum_{j=0}^{\infty}\sum_{p=1}^{\infty}\frac{x^{p+j+1}}{p(p+j+1)}\\
&=\sum_{k=1}^{\infty}\frac{x^{k+1}}{k+1}\sum_{j=1}^{k}\frac{1}{j}=\sum_{k=1}^{\infty}\frac{x^{k+1}H_k}{k+1}.
\end{align*}
Applying
\begin{equation}\label{e:25}
\sum_{k=1}^{\infty}\frac{H_kx^k}{k+1}=\frac{1}{2x}\log^2(1-x),
\end{equation}
which can easily be derived from (\ref{e:23}), we get
\begin{equation}\label{B}
\sum_{p=1}^{\infty}\sum_{j=1}^{\infty}\frac{x^{p+j}}{p(p+j)}=\frac{1}{2}\log^2(1-x).
\end{equation}
On the other hand, we have
\begin{align}\label{C}
\sum_{p=1}^{k}\sum_{j=1}^{\infty}\frac{x^{p+j}}{p(p+j)}=\sum_{p=1}^{k}\frac{1}{p}\sum_{j=p+1}^{\infty}\frac{x^{j}}{j}=-H_k\log(1-x)-\sum_{p=1}^{k}\frac{1}{p}\sum_{j=1}^{p}\frac{x^j}{j}.
\end{align}
Replacing (\ref{B}) and (\ref{C}) in (\ref{A}) we find
\begin{align}\label{D}
\sum_{j=1}^{\infty}\frac{H_jx^j}{j+k+1}&=\frac{\log^2(1-x)}{2x^{k+1}}+\frac{\log(1-x)}{x^{k+1}}\sum_{p=1}^{k}\frac{x^p}{p}\notag\\
&-\frac{H_k\log(1-x)}{x^{k+1}}-\frac{1}{x^{k+1}}\sum_{p=1}^{k}\frac{1}{p}\sum_{j=1}^{p-1}\frac{x^j}{j}.
\end{align}
Combining (\ref{e:22}), (\ref{e:25}) and (\ref{D}), we see that Lemma 1.8 is valid for $|x|<1$. The radius of convergence of the power series $\sum_{j=1}^{\infty}\frac{H_jx^j}{(j+1)(j+k+1)}$ is 1. Clearly, this power series is also convergent at $x=-1$ and $x=1$. So, by Abel's limit theorem,  Lemma 1.8 is also valid for $|x|\leq 1$. Lemma 1.8 provides a generalization of the first part of Corollary 1 in \cite{32} with $p=1$.
\end{proof}
\section{main results}
We prove the following theorems.
\begin{thm} For all $a\in\mathbb{R}\backslash (-2,0]$,and $s$ be a positive real number,  we have
\begin{align}\label{e:26}
\int_{0}^{1}\frac{1-x^s}{1-x}\log(x+a)dx&=(\gamma+\psi(s+1))\log(a+1)-Li_2\left(\frac{1}{a+1}\right)\nonumber\\
&+\sum\limits_{k=1}^\infty\frac{1}{k^2(a+1)^k\binom{s+k}{k}}.
\end{align}
\end{thm}
\begin{proof} We have
\begin{align*}
&\int_{0}^{1}\frac{1-x^s}{1-x}\log(x+a)dx=\int_{0}^{1}\frac{1-x^s}{1-x}\log\left[(a+1)\left(1-\frac{1-x}{1+a}\right)\right]dx\\
&=\log(a+1)\int_{0}^{1}\frac{1-x^s}{1-x}dx+\int_{0}^{1}\frac{1-x^s}{1-x}\log\left(1-\frac{1-x}{1+a}\right)dx.\\
\end{align*}
Using Lemma 1.6 and  the following  well-known identity, the proof is completed.
\begin{equation*}
\int_{0}^{1}\frac{1-x^s}{1-x}dx=\gamma+\psi(s+1),
\end{equation*}
where $\gamma=0.57721...$ is the Euler-Mascheroni constant.
\end{proof}
If we set $s=n\in\mathbb{N}\cup\{0\}$ in (\ref{e:26}), we get the following corollary
\begin{cor}For all $n\in\mathbb{N}\cup\{0\}$ and $a\in\mathbb{R}\backslash (-2,0)$ the following identity holds:
\begin{align}\label{e:27}
\int_{0}^{1}\frac{1-x^n}{1-x}\log(x+a)dx&=H_n\log(a+1)-Li_2\bigg(\frac{1}{a+1}\bigg)\notag\\
&+\sum_{k=1}^{\infty}\frac{1}{k^2(a+1)^k\binom{n+k}{k}}.
\end{align}
\end{cor}
\begin{thm} For  all $n\in\mathbb{N}$ and $a\in\mathbb{R}$ we have
\begin{align}\label{e:28}
\int_{0}^{1}\frac{1-x^n}{1-x}\log(x+a)dx&=H_n\log(a+1)-\sum_{k=1}^{n}\frac{(-a)^k}{k}\sum_{j=1}^{k}\frac{(-1/a)^j}{j}\nonumber\\
&-\log(1+1/a)\sum_{k=1}^{n}\frac{(-a)^k}{k}.
\end{align}
\end{thm}
\begin{proof} Clearly, we have by using Lemma 1.7
\begin{align*}
&\int_{0}^{1}\frac{1-x^n}{1-x}\log(x+a)dx=\sum_{k=0}^{n-1}\int_{0}^{1}x^k\log(x+a)dx\\
&=\log(a+1)\sum_{k=1}^{n}\frac{1}{k}-\sum_{k=1}^{n}\frac{(-a)^{k}}{k}\sum_{j=1}^{k}\frac{(-1/a)^j}{j}\\
&-\log(1+1/a)\sum_{k=1}^{n}\frac{(-a)^k}{k}.
\end{align*}
This completes the proof.
\end{proof}
Equating the right-hand sides of (\ref{e:27}) and (\ref{e:28}), and taking into account Lemma 1.2, we arrive at the following conclusion.
\begin{cor} For all $n\in\mathbb{N}$ and $a\in\mathbb{R}\backslash (-2,0)$ the following equality holds:
\begin{align}\label{e:29}
\sum\limits_{k=1}^\infty\frac{1}{k^2(a+1)^k\binom{n+k}{k}}&=Li_2\left(\frac{1}{a+1}\right)-\log(1+1/a)\sum_{k=1}^{n}\frac{(-a)^k}{k}\notag\\
&-H_n^{(2)}-\sum_{k=1}^{n}\frac{(H_k+H_{n-k}-H_n)(-a)^k}{k}.
\end{align}
\end{cor}
\begin{cor}Let $a\in\mathbb{R}\backslash (-2,0)$ and $n\in\mathbb{N}$. Then we have
\begin{align}\label{e:30}
\sum _{k=1}^{\infty } \frac{1}{k^2(a+1)^k \binom{n+k}{k}}&=-H_n\log(1+1/a)-Li_2\bigg(\frac{1}{a+1}\bigg)-Li_2(-1/a)\notag\\
&-\frac{1}{2}\log^2(1+1/a)-\sum_{k=1}^{\infty}\frac{H_{n+k}(-1/a)^k}{k}.
\end{align}
\end{cor}
\begin{proof} Using the power series $\log(1+t)=\sum_{k=1}^{\infty}\frac{(-1)^{k-1}t^k}{k}$ with $t=1/a$, we get after some simplifications
\begin{align*}
&\log(1+1/a)\sum_{k=1}^{n}\frac{(-a)^k}{k}-H_n^{(2)}-\sum_{k=1}^{n}\frac{(H_k+H_{n-k}-H_n)(-a)^k}{k}\\
&=\log(1+1/a)\sum_{k=1}^{n}\frac{(-a)^k}{k}+\sum_{k=1}^{n}\frac{(-a)^k}{k}\sum_{j=1}^{k}\frac{(-1/a)^j}{j}\\
&=\sum_{k=1}^{n}\frac{(-a)^k}{k}\bigg(\log(1+1/a)+\sum_{j=1}^{k}\frac{(-1/a)^j}{j}\bigg)\\
&=-\sum_{k=1}^{n}\frac{(-a)^k}{k}\sum_{j={k+1}}^{\infty}\frac{(-1/a)^j}{j}.\\
\end{align*}
If we set $j-k=j'$ and drop the prime on $j$, and interchanging the order of the two sums, this is
\begin{align}\label{e:31}
&\log(1+1/a)\sum_{k=1}^{n}\frac{(-a)^k}{k}-H_n^{(2)}-\sum_{k=1}^{n}\frac{(H_k+H_{n-k}-H_n)(-a)^k}{k}\notag\\
&=-\sum_{j=1}^{\infty}\frac{(-1/a)^j}{j}\sum_{k=1}^{n}\bigg(\frac{1}{k}-\frac{1}{k+j}\bigg)\notag\\
&=-\sum_{j=1}^{\infty}\frac{(-1/a)^j}{j}\sum_{k=1}^{n}(H_n-H_{n+j}+H_j)\notag\\
&=-H_n\sum_{j=1}^{\infty}\frac{(-1/a)^j}{j}-\sum_{j=1}^{\infty}\frac{(-1/a)^jH_j}{j}+\sum_{j=1}^{\infty}\frac{(-1/a)^jH_{n+j}}{j}.
\end{align}
Using (\ref{e:23}) with $x=-1/a$ and combining   (\ref{e:29}) and (\ref{e:31}), the assertion follows.

Equating the right-hand sides of (\ref{e:29}) and (\ref{e:30}), we get the following conclusion.
\end{proof}
\begin{cor}Let $a\in\mathbb{R}\backslash [-1,1)$ and $n\in\mathbb{N}$. Then we have
\begin{align}\label{e:32}
&\sum_{k=1}^{\infty}\frac{(-1/a)^kH_{n+k}}{k}=H_n^{(2)}+\log(1+1/a)\sum_{k=1}^{n}\frac{(-a)^k}{k}+\frac{1}{2}\log^2(1+1/a)\notag\\
&+Li_2(-1/a)-H_n\log(1+1/a)+\sum _{k=1}^n \frac{(-a)^k (H_{n-k}+H_k-H_n)}{k}.
\end{align}
\end{cor}
\begin{thm} Let $s\in\mathbb{R}$, which is not a negative integer, and $a>0$. Then we have
\begin{align}\label{e:33}
&\int_{0}^{1}\frac{1-x^s}{1-x}\log^2(x+a)dx=(\gamma+\psi(s+1))\log^2(a+1)+2\zeta(3)\notag\\
&-\log(a+1)\log^2\bigg(\frac{a}{a+1}\bigg)-2\log(a+1)Li_2\bigg(\frac{1}{a+1}\bigg)\notag\\
&+2\log\bigg(\frac{a}{a+1}\bigg)Li_2\bigg(\frac{a}{a+1}\bigg)-2Li_3\bigg(\frac{a}{a+1}\bigg)\notag\notag\\
&+2\log(a+1)\sum_{k=1}^{\infty}\frac{1}{k^2(a+1)^k\binom{s+k}{k}}-\sum_{k=1}^{\infty}\frac{2H_{k-1}}{k^2(a+1)^k\binom{s+k}{k}},
\end{align}
and for $a\geq 2$, and $n\in\mathbb{N}$
\begin{align}\label{e:34}
&\int_{0}^{1}\frac{1-x^n}{1-x}\log^2(a-x)dx=H_n\log^2(a-1)+2\zeta(3)\notag\\
&-2H_n^{(2)}\log(a-1)-2\log(a-1)\sum_{k=1}^{n}\frac{(H_k+H_{n-k}-H_n)a^k}{k}\notag\\
&-\log(a-1)\log^2\bigg(\frac{a}{a-1}\bigg)-\frac{2}{3}\log^3\bigg(\frac{a}{a-1}\bigg)-2Li_3\bigg(\frac{a-1}{a}\bigg)\notag\\
&-2\log\bigg(\frac{a}{a-1}\bigg)Li_2\bigg(\frac{a-1}{a}\bigg)-\sum_{k=1}^{\infty}\frac{2(-1)^kH_{k-1}}{k^2(a-1)^k\binom{n+k}{k}}\notag\\
&+2\log(a-1)\log\bigg(\frac{a}{a-1}\bigg)\sum_{k=1}^{n}\frac{a^k}{k}.
\end{align}
Here $\psi$ is the digamma function, and $\zeta$ is the Riemann zeta function.
\end{thm}
\begin{proof} Clearly we have
\begin{align}\label{e:35}
&\int_{0}^{1}\frac{1-x^s}{1-x}\log^2(x+a)dx=\int_{0}^{1}\frac{1-x^s}{1-x}\log^2\left[(a+1)\left(1-\frac{1-x}{a+1}\right)\right]dx \nonumber\\
&=\log^2(a+1)\int_{0}^{1}\frac{1-x^s}{1-x}dx+\int_{0}^{1}\frac{1-x^s}{1-x}\log^2\left(1-\frac{1-x}{a+1}\right)dx\nonumber\\
&+2\log(a+1)\int_{0}^{1}\frac{1-x^s}{1-x}\log\left(1-\frac{1-x}{a+1}\right)dx.
\end{align}
Using the power series
\begin{equation}\label{e:36}
\log^2(1-t)=2\sum_{k=1}^{\infty}\frac{t^{k+1}H_k}{k+1},
\end{equation}
we get
\begin{align*}
\int_{0}^{1}\frac{1-x^s}{1-x}\log^2\left(1-\frac{1-x}{a+1}\right)dx&=2\int_{0}^{1}\frac{1-x^s}{1-x}\sum_{k=1}^{\infty}\frac{H_k}{k+1}\frac{(1-x)^{k+1}}{(a+1)^{k+1}}dx\\
&=2\int_{0}^{1}(1-x^s)\sum_{k=1}^{\infty}\frac{H_k}{k+1}\frac{(1-x)^{k}}{(a+1)^{k+1}}dx.
\end{align*}
Since
$$
\sum_{k=1}^{\infty}\bigg|\frac{H_k}{k+1}\frac{(1-x^s)(1-x)^{k+1}}{(a+1)^{k+1}}\bigg|\leq \sum_{k=1}^{\infty}\frac{H_k}{(k+1)(a+1)^{k+1}}<\infty,
$$
we can interchanging the order of the summation and integration, and upon simplification,
\begin{align}\label{e:37}
&\int_{0}^{1}\frac{1-x^s}{1-x}\log^2\left(1-\frac{1-x}{a+1}\right)dx\notag\\
&=2\sum_{k=1}^{\infty}\frac{H_k}{(k+1)(a+1)^{k+1}}\int_{0}^{1}(1-x^s)(1-x)^kdx.
\end{align}
Using (\ref{e:20}) with $k\to k+1$ and shifting the index $k$, and then upon simplifying, (\ref{e:37}) can be written as follows:
\begin{align}\label{e:38}
\int_{0}^{1}\frac{1-x^s}{1-x}\log^2\left(1-\frac{1-x}{a+1}\right)dx=\sum_{k=1}^{\infty}\frac{2H_{k-1}}{k^2(a+1)^{k}}-\sum_{k=1}^{\infty}\frac{2H_{k-1}}{k^2(a+1)^{k}\binom{s+k}{k}}.
\end{align}
Thus, using (\ref{e:38}) and Lemma 1.6 in (\ref{e:35}), and noting that $\int_{0}^{1}\frac{1-x^s}{1-x}dx =\gamma+\psi(s+1)$, we arrive at
\begin{align}\label{e:39}
&\int_{0}^{1}\frac{1-x^s}{1-x}\log^2(x+a)dx=(\gamma+\psi(s+1))\log^2(a+1)\notag\\
&+2\log(a+1)\sum_{k=1}^{\infty}\frac{1}{k^2(a+1)^k\binom{s+k}{k}}-2\log(a+1)Li_2\bigg(\frac{1}{a+1}\bigg)\notag\\
&+\sum_{k=1}^{\infty}\frac{2H_{k-1}}{k^2(a+1)^k}-\sum_{k=1}^{\infty}\frac{2H_{k-1}}{k^2(a+1)^k\binom{s+k}{k}}.
\end{align}
According to  \cite{4}, we have, for $0\leq x\leq 1$,
\begin{align}\label{e:40}
2\sum_{k=1}^{\infty}\frac{H_{k-1}}{k^2}x^k&=\log x\log^2(1-x)+2\log(1-x)Li_2(1-x)\notag\\
&-2Li_3(1-x)+2\zeta(3).
\end{align}
Hence, applying this formula with $x=1/(a+1)$ in equation (\ref{e:39}) we complete the proof of (\ref{e:33}). In order to prove (\ref{e:34}) replace $a$ by $-a$ in (\ref{e:33}) and put $s=n\in\mathbb{N}$. Then we obtain
\begin{align}\label{e:41}
&\int_{0}^{1}\frac{1-x^n}{1-x}\left[\log^2(a-x)-\pi^2+2\pi i\log(a-x)\right]dx\notag\\
&=H_n\left(\log^2(a-1)-\pi^2+2\pi i\log(a-1)\right)-2\sum_{k=1}^{\infty}\frac{H_{k-1}}{k^2(a-1)^k\binom{n+k}{k}}\notag\\
&+2(\log(a-1)+\pi i)\sum_{k=1}^{\infty}\frac{1}{k^2(a-1)^k\binom{n+k}{k}}\notag\\
&-2(\log(a-1)+\pi i)Li_2\bigg(\frac{1}{1-a}\bigg)+\sum_{k=1}^{\infty}\frac{2(-1)^kH_{k-1}}{k^2(a-1)^k}.\notag\\
\end{align}
Equating the real parts of each side of this equation one gets
\begin{align}\label{e:42}
&\int_{0}^{1}\frac{1-x^n}{1-x}\log^2(a-x)dx=H_n\log^2(a-1)+\sum_{k=1}^{\infty}\frac{2(-1)^kH_{k-1}}{k^2(a-1)^k}\notag\\
&-2\log(a-1)Li_2\left(\frac{1}{1-a}\right)-\sum_{k=1}^{\infty}\frac{2(-1)^kH_{k-1}}{k^2(a-1)^k\binom{n+k}{k}}\notag\\
&+2\log(a-1)\sum_{k=1}^{\infty}\frac{1}{k^2(1-a)^k\binom{n+k}{k}}.
\end{align}
Replacing $a$ by $-a$ in (\ref{e:29}) and using Lemma 1.2 we get for $a\in\mathbb{C}\backslash (0,2)$ and $n\in\mathbb{N}$
\begin{align}\label{e:43}
&\sum_{k=1}^{\infty}\frac{1}{k^2(1-a)^k\binom{n+k}{k}}=Li_2\bigg(\frac{1}{1-a}\bigg)+\log\bigg(\frac{a}{a-1}\bigg)\sum_{k=1}^{n}\frac{a^k}{k}\notag\\
&-H_n^{(2)}-\sum_{k=1}^{n}\frac{(H_k+H_{n-k}-H_n)a^k}{k}.
\end{align}
On the other hand, we know from \cite{10} that for $0\leq x\leq 1$

\begin{align}\label{e:44}
&\sum_{k=1}^{\infty}\frac{2(-1)^kH_{k-1}x^k}{k^2}=\log x\log^2(1+x)-\frac{2}{3}\log^3(1+x)\notag\\
&-2Li_3\bigg(\frac{1}{1+x}\bigg)-2\log(1+x)Li_2\left(\frac{1}{1+x}\right)+2\zeta(3).
\end{align}
Using (\ref{e:43}) and (\ref{e:44}) with $x=1/(1-a)$ in (\ref{e:42}) we see that  (\ref{e:34}) is valid.
\end{proof}
If we set $s=n\in\mathbb{N}$ in (\ref{e:33}) and use (\ref{e:10}) with $x=1/(a+1)$ and Corollary 2.4 and Lemma 1.2, we get the following conclusion.
\begin{cor}  For all $a\in\mathbb{R}\backslash (-2,0)$,and $n\in\mathbb{N}$, we have
\begin{align}\label{e:45}
&\int_{0}^{1}\frac{1-x^n}{1-x}\log^2(x+a)dx=2\bigg(\zeta(2)-Li_2\bigg(\frac{1}{a+1}\bigg)\bigg)\log\bigg(\frac{a}{a+1}\bigg)\notag\\
&+\log(a+1)\log^2\bigg(\frac{a}{a+1}\bigg)-2Li_3\bigg(\frac{a}{a+1}\bigg)+2\zeta(3)+H_n\log^2(a+1)\notag\\
&-2\log(a+1)\log(1+1/a)\sum_{k=1}^{n}\frac{(-a)^k}{k}-2H_n^{(2)}\log(a+1)\notag\\
&-2\sum_{k=1}^{\infty}\frac{H_{k-1}}{k^2(a+1)^{k}\binom{n+k}{k}}-2\log(a+1)\sum_{k=1}^{n}\frac{(H_k+H_{n-k}-H_n)(-a)^k}{k}.
\end{align}
\end{cor}
\begin{thm} For all $n\in\mathbb{N}$ and $a\in\mathbb{R}\backslash[-1,1)$ the following equality holds:
\begin{align}\label{e:46}
&\int_{0}^{1}\frac{1-x^n}{1-x}\log^2(x+a)dx=H_n\log^2(a+1)-2\log(a+1)H_n^{(2)}\notag\\
&-2\log(a+1)\sum_{k=1}^{n}\frac{(H_k+H_{n-k}-H_n)(-a)^k}{k}\notag\\
&+2\log(1+1/a)\sum_{k=1}^{n}\frac{(-a)^kH_k}{k}+\left(\log^2a-\log^2(a+1)\right)\sum_{k=1}^{n}\frac{(-a)^k}{k}\notag\\
&+2\sum_{p=0}^{n-1}\bigg\{\sum_{k=1}^{n-p}\bigg(\frac{H_{p+k}-H_k}{(p+k)k}+\frac{1}{k^2(p+k)}\bigg)\bigg\}(-a)^p.
\end{align}
\end{thm}
\begin{proof} Clearly we have
\begin{align}\label{e:47}
&\int_{0}^{1}\frac{1-x^n}{1-x}\log^2(x+a)dx=\sum_{k=0}^{n-1}\int_{0}^{1}x^k\log^2(x+a)dx\notag\\
&=\sum_{k=0}^{n-1}\int_{0}^{1}x^k\left(\log a+\log(1+x/a)\right)^2dx.
\end{align}
Expanding the last term in the brackets, after a simple computation, we get
\begin{align}\label{e:48}
&\int_{0}^{1}\frac{1-x^n}{1-x}\log^2(x+a)dx=2\log a\sum_{k=0}^{n-1}\int_{0}^{1}x^k\log(x+a)dx\notag\\
&-H_n\log^2a+\sum_{k=0}^{n-1}\int_{0}^{1}x^k\log^2\bigg(1+\frac{x}{a}\bigg)dx.
\end{align}
From Lemma 1.7 it follows that
\begin{align*}
\sum_{k=0}^{n-1}\int_{0}^{1}x^k\log (x+a)dx &=\sum_{k=0}^{n-1}\bigg[\frac{\log(a+1)}{k+1}-\frac{(-a)^{k+1}}{k+1}\sum_{j=0}^{k+1}\frac{(-1/a)^j}{j}\\
&-\frac{(-a)^{k+1}}{k+1}\log(1+1/a)\bigg].
\end{align*}
After simplifying, this becomes
\begin{align}\label{e:49}
\sum_{k=0}^{n-1}\int_{0}^{1}x^k\log (x+a)dx&=H_n\log(a+1)-\sum_{k=1}^{n}\frac{(-a)^{k}}{k}\sum_{j=1}^{k}\frac{(-1/a)^j}{j}\notag\\
&-\log(1+1/a)\sum_{k=1}^{n}\frac{(-a)^{k}}{k}.
\end{align}
By the help of (\ref{e:36}), we can easily show that
\begin{align*}
\int_{0}^{1}x^k\log^2\left(1+\frac{x}{a}\right)dx=2\int_{0}^{1}x^k\sum_{j=1}^{\infty}\frac{(-x/a)^{j+1}H_j}{j+1}dx.
\end{align*}
Reversing the order of the summation and integration, we get, after an easy computation
\begin{align}\label{e:50}
\int_{0}^{1}x^k\log^2\left(1+\frac{x}{a}\right)dx=2\sum_{j=1}^{\infty}\frac{(-1/a)^{j+1}H_j}{(j+1)(k+j+2)}.
\end{align}
Summing both sides of this equation from $k=0$ to $k=n-1$, and using Lemma 1.8 with $x=-1/a$ and $k\to k+1$, we get
\begin{align}\label{e:51}
&\sum_{k=0}^{n-1}\int_{0}^{1}x^k\log^2\left(1+\frac{x}{a}\right)dx=2\sum_{k=0}^{n-1}\sum_{j=1}^{\infty}\frac{(-1/a)^{j+1}H_j}{(j+1)(k+j+2)}\notag\\
&=\bigg(H_n-\sum_{k=1}^{n}\frac{(-a)^k}{k}\bigg)\log^2(1+1/a)\notag\\
&+2\log(1+1/a)\bigg(\sum_{k=1}^{n}\frac{(-a)^kH_k}{k}-\sum_{k=1}^{n}\frac{(-a)^k}{k}\sum_{j=1}^{k}\frac{(-1/a)^j}{j}\bigg)\notag\\
&+2\sum_{k=1}^{n}\frac{(-a)^k}{k}\sum_{j=1}^{k}\frac{1}{j}\sum_{p=1}^{j}\frac{(-1/a)^p}{p}.
\end{align}
Substituting (\ref{e:49}) and (\ref{e:51}) in (\ref{e:48}), and taking into account Lemma 1.2 and Lemma 1.4 the proof of Theorem 2.9 follows.
\end{proof}
Replacing $a$ by $-a$ in (\ref{e:45}), and equating the real parts of each side of the resulting equation we get
\begin{cor} Let $n\in\mathbb{N}$ and $a\geq 2$. Then we have
\begin{align}\label{e:52}
&\int_{0}^{1}\frac{1-x^n}{1-x}\log^2(a-x)dx=H_n\log^2(a-1)-2 H_n^{(2)}\log(a-1)\notag\\
&-2\log(a-1)\sum_{k=1}^{n}\frac{(H_k+H_{n-k}-H_n)a^k}{k}+2\log(1-1/a)\sum_{k=1}^{n}\frac{H_ka^k}{k}\notag\\
&+2\sum_{p=0}^{n-1}\bigg\{\sum_{k=1}^{n-p}\bigg(\frac{H_{p+k}-H_k}{(p+k)k}+\frac{1}{k^2(p+k)}\bigg)\bigg\}a^p\notag\\
&+\left(\log^2a-\log^2(a-1)\right)\sum_{k=1}^{n}\frac{a^k}{k}.
\end{align}
\end{cor}
Equating the right-hand sides of  (\ref{e:45}) and  (\ref{e:46}), and setting $t=1/(a+1)$ we obtain, in the light of Lemma 1.4, the following conclusion.
\begin{thm} For $n\in\mathbb{N}$ and $0<t\leq 1$, we have
\begin{align}\label{e:53}
&\sum_{k=1}^{\infty}\frac{H_{k-1}t^k}{k^2\binom{n+k}{k}}=\zeta(3)-Li_3(1-t)-\frac{1}{2}\log^2(1-t)\sum_{k=1}^{n}\frac{1}{k}\bigg(\frac{t-1}{t}\bigg)^k\notag\\
&+(\zeta(2)-Li_2(t))\log(1-t)+\log(1-t)\sum_{k=1}^{n}\frac{H_k}{k}\bigg(\frac{t-1}{t}\bigg)^k\notag\\
&-\sum_{p= 0}^{n-1}\bigg\{\sum_{k=1}^{n-p}\bigg(\frac{H_{p+k}-H_k}{(p+k)k}+\frac{1}{k^2(p+k)}\bigg)\bigg\}\bigg(\frac{t-1}{t}\bigg)^p\notag\\
&-\frac{1}{2}\log t\log^2(1-t).
\end{align}
\end{thm}
\begin{thm}For $n\in\mathbb{N}$ and $0 <t\leq1$ the following identity holds.
\begin{align}\label{e:54}
&\sum_{k=1}^{\infty}\frac{H_{k-1}t^k}{k\binom{n+k}{k}}=\frac{1}{2}\log^2(1-t)+\frac{t\log(1-t)}{1-t}\sum_{k=1}^{n}\frac{1}{k}\left(\frac{t-1}{t}\right)^k\notag\\
&+\frac{\log^2(1-t)}{2(1-t)}\sum_{k=1}^{n}\left(\frac{t-1}{t}\right)^k-\frac{t}{1-t}\sum_{k=1}^{n}\frac{H_k}{k}\left(\frac{t-1}{t}\right)^k\notag\\
&+\frac{1}{1-t}\sum_{p=0}^{n-1}\bigg\{\sum_{k=1}^{n-p}\bigg(\frac{H_{p+k}-H_k}{(p+k)k}+\frac{1}{k^2(p+k)}\bigg)\bigg\}p\bigg(\frac{t-1}{t}\bigg)^p\notag\\
&-\frac{\log(1-t)}{1-t}\sum_{k=1}^{n}H_k\left(\frac{t-1}{t}\right)^k.
\end{align}
\end{thm}
\begin{proof}
The radius of convergence of the power series $\sum_{k=0}^{\infty}\frac{H_{k-1}t^k}{k^2\binom{n+k}{k}}$ is 1. Let $t\in(0,1)$. Then this power series is uniformly convergent on $[0,t]$. Therefore, it can be differentiated term-by-term at $t$ (see \cite[Theorem 7.4.4(5)]{Ter}). Differentiating both sides of (\ref{e:53}) with respect to $t$, and then multiplying the result by $t$, we get (\ref{e:54}). But since this series is also convergent at $t=1$, (\ref{e:54}) is valid for $t=1$ by Abel's limit theorem. The proof is completed.
\end{proof}
\begin{thm} For $0<t<1$ if $n=1$ and $0<t\leq1$ if $n\geq2$ the following identity holds.
\begin{align}\label{e:55}
&\sum_{k=1}^{\infty}\frac{H_{k-1}t^k}{\binom{n+k}{k}}=f_0(t)+f_1(t)\sum_{k=1}^{n}\frac{1}{k}\left(\frac{t-1}{t}\right)^k-f_2(t)\sum_{k=1}^{n}\left(\frac{t-1}{t}\right)^k\notag\\
&-f_3(t)\sum_{k=1}^{n}k\left(\frac{t-1}{t}\right)^k+f_4(t)\sum_{k=1}^{n}H_k\left(\frac{t-1}{t}\right)^k\notag\\
&+f_5(t)\sum_{k=1}^{n}kH_k\left(\frac{t-1}{t}\right)^k-\frac{t}{(1-t)^2}\sum_{k=1}^{n}\frac{H_k}{k}\left(\frac{t-1}{t}\right)^k\notag\\
&+\frac{t}{(1-t)^2}\sum_{p=0}^{n-1}\bigg\{\sum_{k=1}^{n-p}\bigg(\frac{H_{p+k}-H_k}{(p+k)k}+\frac{1}{k^2(p+k)}\bigg)\bigg\}p\bigg(\frac{t-1}{t}\bigg)^p\notag\\
&-\frac{1}{(1-t)^2}\sum_{p=0}^{n-1}\bigg\{\sum_{k=1}^{n-p}\bigg(\frac{H_{p+k}-H_k}{(p+k)k}+\frac{1}{k^2(p+k)}\bigg)\bigg\}p^2\bigg(\frac{t-1}{t}\bigg)^p.
\end{align}
Here
\begin{align*}
&f_0(t)=-\frac{t\log(1-t)}{1-t},\quad f_1(t)=\frac{t\log(1-t)-t^2}{(1-t)^2}\\
&f_2(t)=\frac{4t\log(1-t)-t\log^2(1-t)}{2(1-t)^2},\quad f_3(t)=\frac{\log^2(1-t)}{2(1-t)^2},\\
&f_4(t)=\frac{2t-t\log(1-t)}{(1-t)^2},\quad f_5(t)=\frac{\log(1-t)}{(1-t)^2}.
\end{align*}
\begin{proof}
Differentiating both sides of (\ref{e:54}) with respect to $t$, which can be justified in exactly the same way with the proof of Theorem 2.12, and then multiplying the resulting identity by $t$, we complete the proof.
\end{proof}
\end{thm}
Equating the right-hand sides of (\ref{e:34}) and (\ref{e:52}) and setting $t=1/(a-1)$ we get:
\begin{thm} For $n\in\mathbb{N}$ and $0 <t\leq 1$ we have
\begin{align}\label{e:56}
&\sum_{k=1}^{\infty}\frac{(-t)^kH_{k-1}}{k^2\binom{n+k}{k}}=\frac{1}{2}\log t\log^2(t+1)-\frac{1}{3}\log^3(1+t)\notag\\
&-\log(1+t)Li_2\bigg(\frac{1}{1+t}\bigg)-Li_3\bigg(\frac{1}{1+t}\bigg)+\zeta(3)\notag\\
&+\log(1+t)\sum_{k=1}^{n}\frac{H_k}{k}\bigg(\frac{t+1}{t}\bigg)^k-\frac{1}{2}\log^2(1+t)\sum_{k=1}^{n}\frac{1}{k}\bigg(\frac{t+1}{t}\bigg)^k\notag\\
&-\sum_{p=0}^{n-1}\bigg\{\sum_{k=1}^{n-p}\bigg(\frac{H_{p+k}-H_k}{(p+k)k}+\frac{1}{k^2(p+k)}\bigg)\bigg\}\bigg(\frac{t+1}{t}\bigg)^p.
\end{align}
\end{thm}

\begin{thm}For $n\in\mathbb{N}$ and $0 <t\leq 1$ we have
\begin{align}\label{e:57}
&\sum_{k=1}^{\infty}\frac{(-t)^kH_{k-1}}{k\binom{n+k}{k}}=\frac{1}{2} \log ^2(t+1)+\frac{t}{t+1}
\sum_{k=1}^n\frac{H_k}{k}\left(\frac{t+1}{t}\right)^k\notag\\
&-\frac{t\log (t+1)}{t+1} \sum _{k=1}^n \frac{1}{k}\left(\frac{t+1}{t}\right)^k+\frac{\log^2(t+1)}{2(t+1)} \sum _{k=1}^n \left(\frac{t+1}{t}\right)^k\notag\\
&+\frac{1}{t+1}\sum_{p=0}^{n-1}\bigg\{\sum_{k=1}^{n-p}\bigg(\frac{H_{p+k}-H_k}{(p+k)k}+\frac{1}{k^2(p+k)}\bigg)\bigg\}p\bigg(\frac{t+1}{t}\bigg)^p\notag\\
&-\frac{\log (t+1)}{t+1}\sum _{k=1}^n H_k \left(\frac{t+1}{t}\right)^k.
\end{align}
\end{thm}
\begin{proof}
Differentiating both sides of (\ref{e:56}) with respect to $t$, which can be justified in exactly the same way with the proof of Theorem 2.12, and then multiplying by $t$, we complete the proof.
\end{proof}
\begin{thm}For $n\in\mathbb{N}$ and $0 <t\leq 1$ we have
\begin{align}\label{e:58}
&\sum_{k=1}^{\infty}\frac{(-t)^kH_{k-1}}{\binom{n+k}{k}}=\frac{t\log (t+1)}{t+1}+\frac{t}{(t+1)^2}\sum _{k=1}^n \frac{H_k}{k} \left(\frac{t+1}{t}\right)^k\notag\\
&-\frac{t^2+t\log (t+1)}{(t+1)^2}\sum _{k=1}^n \frac{1}{k}\left(\frac{t+1}{t}\right)^k-\frac{\log ^2(t+1)}{2(t+1)^2}\sum _{k=1}^n k \left(\frac{t+1}{t}\right)^k\notag\\
&-\frac{t}{(t+1)^2}\sum _{p=0}^{n-1}\bigg\{\sum _{k=1}^{n-p} \left(\frac{H_{k+p}-H_k}{k (k+p)}+\frac{1}{k^2 (k+p)}\right)\bigg\}p\left(\frac{t+1}{t}\right)^p\notag\\
&-\frac{1}{(t+1)^2}\sum _{p=0}^{n-1}\bigg\{\sum _{k=1}^{n-p} \left(\frac{H_{k+p}-H_k}{k (k+p)}+\frac{1}{k^2 (k+p)}\right)\bigg\}p^2\left(\frac{t+1}{t}\right)^p\notag\\
&-\frac{2t-t\log (t+1)}{(t+1)^2}\sum _{k=1}^n H_k \left(\frac{t+1}{t}\right)^k+\frac{\log (t+1)}{(t+1)^2}\sum _{k=1}^n k H_k \left(\frac{t+1}{t}\right)^k\notag\\
&+\frac{4 t\log (t+1)-t\log ^2(t+1)}{2 (t+1)^2}\sum _{k=1}^n \left(\frac{t+1}{t}\right)^k.
\end{align}
\end{thm}
\begin{proof}
Differentiating both sides of (\ref{e:57}) with respect to $t$, which can be justified in exactly the same way with the proof of Theorem 2.12, and then multiplying the resulting identity by $t$, we complete the proof.
\end{proof}
Setting $t=1/(a+1)$ in (\ref{e:29}), we get:
\begin{thm} For $n\in\mathbb{N}$ and $|t|\leq 1$ we have
\begin{align}\label{e:59}
\sum_{k=1}^{\infty}\frac{t^k}{k^2\binom{n+k}{k}}&=Li_2(t)-H_n^{(2)}+\log(1-t)\sum_{k=1}^{n}\frac{1}{k}\left(\frac{t-1}{t}\right)^k\notag\\
&-\sum_{k=1}^{n}\frac{H_k+H_{n-k}-H_n}{k}\left(\frac{t-1}{t}\right)^k.
\end{align}
\end{thm}
\begin{thm} For $n\in\mathbb{N}$ and $|t| \leq 1$, we have
\begin{align}\label{e:60}
\sum_{k=1}^{\infty}\frac{t^k}{k\binom{n+k}{k}}&=\frac{1}{1-t} \sum _{k=1}^n \left(H_k+H_{n-k}-H_n\right)\bigg(\frac{t-1}{t}\bigg)^{k}\notag\\
&+\sum _{k=1}^n \frac{1}{k}\bigg(\frac{t-1}{t}\bigg)^{k-1}-\bigg(\frac{t-1}{t}\bigg)^n \log (1-t).
\end{align}
\end{thm}
\begin{proof}
\end{proof}
\begin{thm}For $-1\leq t<1$ if $n=1$ and $|t|\leq 1$ if $n\geq 2$ we have
\begin{align}\label{e:61}
&\sum_{k=1}^{\infty}\frac{t^k}{\binom{n+k}{k}}=\frac{1}{(1-t)^2}\sum _{k=1}^n(t-k)\left(H_k+H_{n-k}-H_{n}\right)\left(\frac{t-1}{t}\right)^k\notag\\
&-\frac{1}{t}\sum _{k=0}^{n-2} \frac{1}{k+2}\left(\frac{t-1}{t}\right)^k+ \frac{2t+n\log(1-t)}{1-t}\left(\frac{t-1}{t}\right)^n+1.
\end{align}
\end{thm}
\begin{proof}
Differentiating both sides of (\ref{e:60}) with respect to $t$, which can be justified in exactly the same way with the proof of Theorem 2.12, and then multiplying the resulting identity by $t$, we complete the proof.
\end{proof}
In the next theorem we provide a generalization of identity  (\ref{e:11}).
\begin{thm}Let $n\in\mathbb{N}$ and $t\in\mathbb{R}$, which is not a negative integer. Then
\begin{equation}\label{e:62}
\psi^{(n)}(t+1)=(-1)^{n+1}n!\sum_{k=1}^{\infty}\frac{|s(k,n)|}{kk!\binom{t+k}{k}},
\end{equation}
where $s(n,k)$ are Stirling numbers of the first kind.
\end{thm}
\begin{proof} We assume that $t>0$. Making the change of variable $x=e^{-t}$ in (\ref{e:3}) yields for $t>-1$
\begin{equation}\label{e:63}
\psi^{(n)}(t+1)=(-1)^{n+1}\int_{0}^{1}\frac{x^t\left(\log(1-(1-x))\right)^n}{1-x}dx,
\end{equation}
Using (see \cite[pg. 351]{12})
$$
\left(\log(1-x)\right)^n=(-1)^nn!\sum_{k=1}^{\infty}\frac{|s(k,n)|x^k}{k!},\quad |x|<1,
$$
we conclude from  (\ref{e:63}) for $t>0$   and $|x|<1$
\begin{align}\label{e:64}
\psi^{(n)}(t+1)=(-1)^{n+1}n!\int_{0}^{1}\sum_{k=1}^{\infty}\frac{|s(k,n)|x^{t}(1-x)^{k-1}}{k!}dx\notag\\
=(-1)^{n+1}n!\int_{0}^{1}\sum_{k=1}^{\infty}\frac{|s(k,n)|(1-x)^{t}x^{k-1}}{k!}dx
\end{align}
Let $u\in(0,1)$. Then for all $x\in[0,u]$, we have
$$
|(1-x)^tx^{k-1}|\leq|x|^{k-1}<|u|^{k-1}.
$$
But since this is also valid for $x=0$ and $x=1$, we have, for all $x\in[0,1]$,
$$
|(1-x)^tx^{k-1}|\leq|x|^{k-1}<|u|^{k-1}.
$$
We therefore get for all $x\in[0,1]$
\begin{align*}
\sum_{k=1}^{\infty}\bigg|\frac{s(k,n)(1-x)^{t}x^{k-1}}{k!}\bigg|\leq \sum_{k=1}^{\infty}\frac{|s(k,n)|u^{k-1}}{k!}<\infty,
\end{align*}
since $|u|<1$. So, by Weierstrass $M$-test the series $\sum_{k=1}^{\infty}\frac{|s(k,n)|x^{t}(1-x)^{k-1}}{k!}$ is uniformly convergent on $[0,1]$. Hence, we can reverse the order of integration and summation in (\ref{e:64}) (see \cite[Theorem 7.2.4, p.192]{Ter}), and
\begin{align}\label{e:65}
\psi^{(n)}(t+1)=(-1)^{n+1}n!\sum_{k=1}^{\infty}\frac{|s(k,n)|}{k!}\int_{0}^{1}x^{t}(1-x)^{k-1}dx.
\end{align}
Thus, by the help of (\ref{e:5}), we get
\begin{align*}
\psi^{(n)}(t+1)&=(-)^{n+1}n!\sum_{k=0}^{\infty}\frac{|s(k,n)|}{k!}B(t+1,k)\\
&=(-)^{n+1}n!\sum_{k=0}^{\infty}\frac{|s(k,n)|}{k!}\frac{\Gamma(t+1)\Gamma(k)}{\Gamma(t+k+1)}\\
&=(-)^{n+1}n!\sum_{k=0}^{\infty}\frac{|s(k,n)|}{kk!\binom{t+k}{k}},
\end{align*}
which is the desired result for $t>0$. If $t\leq 0$, but not an integer, making the change of variable $y=e^t$ in (\ref{e:64}), we complete the proof by proceeding exactly as in the case $t>0$.
\end{proof}
\begin{rem}
Formula (\ref{e:62}) is a new series representation for the polygamma functions and it  establishes  an interesting relation between the polygamma functions and the Stirling numbers of the first kind.
\end{rem}
\section{applications}
In this section we offer some illustrative examples based on our main results.
%***************************************************************************************************Example 1
\begin{exam}Setting $t=1$ in (\ref{e:59}) gives for $n\in\mathbb{N}$
\begin{align*}
\sum_{k=1}^{\infty}\frac{1}{k^2\binom{n+k}{k}}=\zeta(2)-H_n^{(2)}.
\end{align*}
\end{exam}
%***************************************************************************************************Example 2
\begin{exam}Setting $t=1$ in (\ref{e:53}) gives for $n\in\mathbb{N}$
\begin{align*}
\sum_{k=1}^{\infty}\frac{H_{k-1}}{k^2\binom{n+k}{k}}=\zeta(3)-H_n^{(3)}.
\end{align*}
\end{exam}
%***************************************************************************************************Example 3
Setting $t=1/2$ in (\ref{e:53}) and using the identities (see \cite{15})

\begin{equation}\label{e:66}
Li_2(1/2)=\frac{\zeta(2)}{2}-\frac{\log^22}{2},
\end{equation}
and
\begin{equation}\label{e:67}
Li_3(1/2)=\frac{1}{24}\left(-12\zeta(2)\log2+4\log^32+21\zeta(3)\right)
\end{equation}
we obtain for $n\in\mathbb{N}$
\begin{exam}
\begin{align}\label{e:68}
&\sum_{k=1}^{\infty}\frac{H_{k-1}}{2^kk^2\binom{n+k}{k}}=-\frac{\log^32}{6}+\frac{\zeta(3)}{8}-\frac{\log ^22}{2}\sum _{k=1}^n \frac{(-1)^k}{k}\notag\\
&-\sum _{p=0}^{n-1} \bigg\{\sum _{k=1}^{n-p} \left(\frac{H_{k+p}-H_k}{k (k+p)}+\frac{1}{k^2 (k+p)}\right)\bigg\}(-1)^p-\log 2 \sum _{k=1}^n \frac{(-1)^k H_k}{k}.\notag\\
\end{align}
\end{exam}
For the particular value  $n=5$ in (\ref{e:68}) we get
%***************************************************************************************************Example 4
\begin{exam}
\begin{align*}
&\sum_{k=1}^{\infty}\frac{H_{k-1}}{2^kk^2\binom{k+5}{k}}=\frac{\zeta (3)}{8}-\frac{\log ^32}{6}+\frac{47 \log ^22}{120}+\frac{2869 \log 2}{3600}-\frac{179693}{216000}.
\end{align*}
\end{exam}
If we let  $n\to \infty$ in (\ref{e:53}), the left-hand side of it goes to zero. Thus, we get for $\frac{1}{2}\leq t<1$:
%***************************************************************************************************Example 5
\begin{exam}Let $n$ be a positive integer. Then
\begin{align}\label{e:69}
&\sum _{p=0}^{\infty}\sum _{k=1}^{\infty} \left(\frac{H_{k+p}-H_k}{k (k+p)}+\frac{1}{k^2 (k+p)}\right)\bigg(\frac{t-1}{t}\bigg)^p=\zeta (2) \log (1-t)\notag\\
&+Li_2\left(\frac{t-1}{t}\right) \log (1-t)-Li_3(1-t)-Li_2(t) \log (1-t)+\zeta(3)\notag\\
&-\log t \log ^2(1-t)-\frac{1}{2} \log ^2t \log (1-t).
\end{align}
\end{exam}
%***************************************************************************************************Example 6
For the particular value  $t=1/2$ of (\ref{e:69}) we get
\begin{exam}Let $n$ be a positive integer. Then
\begin{align*}
\sum _{p=0}^{\infty}\sum _{k=1}^{\infty}\bigg(\frac{H_{k+p}-H_k}{k (k+p)}+\frac{1}{k^2 (k+p)}\bigg)(-1)^p=\frac{5 \log ^32}{6}+\frac{\zeta(2)\log 2}{2}-\frac{\zeta (3)}{8}.
\end{align*}
\end{exam}
For $n\in\mathbb{N}$ and $t=1$ in (\ref{e:54}) we get after straightforward  but lengthly calculation:
%***************************************************************************************************Example 7
\begin{exam}Let $n$ be a positive integer. Then
\begin{align}\label{e:70}
&\sum_{k=1}^{\infty}\frac{H_{k-1}}{k\binom{n+k}{k}}=\frac{1}{n^2}.
\end{align}
\end{exam}
Dividing both sides of (\ref{e:70}) by $n^{m}$ and then summing the resulting equation from $n=1$ to $n=\infty$, one gets
%***************************************************************************************************Example 8
\begin{exam} For $m=0,1,2,3,..$ we have
\begin{align*}
&\zeta(m+2)=\sum_{n=1}^{\infty}\sum_{k=1}^{\infty}\frac{H_{k-1}}{kn^{m}\binom{n+k}{k}},
\end{align*}
which is a new double series representation for the Riemann zeta function $\zeta$.
\end{exam}
For $n\in\mathbb{N}$ and $t=1/2$ in (\ref{e:54}) we get after straightforward  but lengthly calculations:
%***************************************************************************************************Example 9
\begin{exam} Let $n$ be a positive integer. Then
\begin{align*}
&\sum_{k=1}^{\infty}\frac{H_{k-1}}{k2^k\binom{n+k}{k}}=\frac{\log^22}{2}-\log 2\sum_{k=1}^{n}\frac{(-1)^k}{k}+\log^22\sum_{k=1}^{n}(-1)^k\notag\\
&-\sum_{k=1}^{n}\frac{H_k(-1)^k}{k}+2\sum_{p=0}^{n-1}\bigg\{\sum_{k=1}^{n-p}\bigg(\frac{H_{p+k}-H_k}{(p+k)k}+\frac{1}{k^2(p+k)}\bigg)\bigg\}p(-1)^p\notag\\
&+2\log2\sum_{k=1}^{n}H_k(-1)^k.
\end{align*}
\end{exam}
For $n\geq 2$ and $t=1$ in (\ref{e:55}) we get after straightforward  but lengthly calculation:
%***************************************************************************************************Example 10
\begin{exam}Let $n$ be a positive integer. Then
\begin{align*}
&\sum_{k=1}^{\infty}\frac{H_{k-1}}{\binom{n+k}{k}}=\frac{2 n-1}{n (n-1)^2}.
\end{align*}
\end{exam}
Setting $t=1$ in (\ref{e:56}) and using the identities (\ref{e:66}) and (\ref{e:67}) we get for $n\in\mathbb{N}$
%***************************************************************************************************Example 11
\begin{exam}
\begin{align}\label{e:71}
&\sum_{k=1}^{\infty}\frac{(-1)^kH_{k-1}}{k^2\binom{n+k}{k}}=\frac{\zeta(3)}{8}+\log2\sum_{k=1}^{n}\frac{H_k2^k}{k}-\frac{\log^22}{2}\sum_{k=1}^{n}\frac{2^k}{k}\notag\\
&-\sum_{p=0}^{n-1}\bigg\{\sum_{k=1}^{n-p}\bigg(\frac{H_{p+k}-H_k}{(p+k)k}+\frac{1}{k^2(p+k)}\bigg)\bigg\}2^p.
\end{align}
\end{exam}
%***************************************************************************************************Example 12
\begin{exam} For $n=5$ in (\ref{e:71}) we arrive at the following result:
\begin{align*}
&\sum_{k=1}^{\infty}\frac{(-1)^kH_{k-1}}{k^2\binom{k+5}{k}}=\frac{\zeta (3)}{8}-\frac{4060943}{216000}-\frac{128\log ^22}{15}+\frac{7388 \log 2}{225}.
\end{align*}
\end{exam}
%***************************************************************************************************Example 13
\begin{exam} For $n\in\mathbb{N}$ and $t=1/2$ in (\ref{e:56}) we get
\begin{align*}
&\sum_{k=1}^{\infty}\frac{(-1)^kH_{k-1}}{k^22^k\binom{n+k}{k}}=-\frac{\log2\log^2(3/2)}{2}-\frac{\log^3(3/2)}{3}\notag\\
&-\log(3/2)Li_2(2/3)+\zeta(3)-Li_3(2/3)-\frac{\log^2(3/2)}{2}\sum_{k=1}^{n}\frac{3^k}{k}\notag\\
&-\sum_{p=0}^{n-1}\bigg\{\sum_{k=1}^{n-p}\bigg(\frac{H_{p+k}-H_k}{(p+k)k}+\frac{1}{k^2(p+k)}\bigg)\bigg\}3^p.
\end{align*}
\end{exam}
%***************************************************************************************************Example 14
\begin{exam} For $n\in\mathbb{N}$ and $t=1$ in (\ref{e:57}) we find
\begin{align}\label{e:72}
&\sum_{k=1}^{\infty}\frac{(-1)^kH_{k-1}}{k\binom{n+k}{k}}=\frac{\log^22}{2}-\frac{\log 2}{2}\sum_{k=1}^{n}\frac{2^k}{k}+\frac{\log^22}{4}\sum_{k=1}^{n}2^k\notag\\
&+\frac{1}{2}\sum_{k=1}^{n}\frac{H_k2^k}{k}+\frac{1}{2}\sum_{p=0}^{n-1}\bigg\{\sum_{k=1}^{n-p}\bigg(\frac{H_{p+k}-H_k}{(p+k)k}+\frac{1}{k^2(p+k)}\bigg)\bigg\}p2^p\notag\\
&-\frac{\log2}{2}\sum_{k=1}^{n}H_k2^k.
\end{align}
\end{exam}
%***************************************************************************************************Example 15
\begin{exam} Setting $n=5$ in (\ref{e:72})
\begin{align*}
&\sum_{k=1}^{\infty}\frac{(-1)^kH_{k-1}}{k\binom{k+5}{k}}=16 \log ^22-\frac{1096 \log 2}{15}+\frac{77357}{1800}.
\end{align*}
\end{exam}
%***************************************************************************************************Example 16
\begin{exam} For $n\in\mathbb{N}$ and $t=1/2$ in (\ref{e:57}) we obtain
\begin{align*}
&\sum_{k=1}^{\infty}\frac{(-1)^kH_{k-1}}{k2^k\binom{n+k}{k}}=\frac{\log^2(3/2)}{2}+\frac{1}{3}\sum_{k=1}^{n}\frac{H_k3^k}{k}\notag\\
&-\frac{\log(3/2)}{3}\sum_{k=1}^{n}\frac{3^k}{k}+\frac{\log^2(3/2)}{3}\sum_{k=1}^{n}3^k-\frac{2\log(3/2)}{3}\sum_{k=1}^{n}H_k3^k\notag\\
&+\frac{2}{3}\sum_{p=0}^{n-1}\bigg\{\sum_{k=1}^{n-p}\bigg(\frac{H_{p+k}-H_k}{(p+k)k}+\frac{1}{k^2(p+k)}\bigg)\bigg\}3^p.
\end{align*}
\end{exam}
%***************************************************************************************************Example 17
\begin{exam} For $n\in\mathbb{N}$ and $t=1$ in (\ref{e:58}) we get
\begin{align*}
&\sum_{k=1}^{\infty}\frac{(-1)^kH_{k-1}}{\binom{n+k}{k}}=\frac{1}{4}\sum_{k=1}^{n}\frac{H_k2^k}{k}-\frac{1+\log2}{4}\sum_{k=1}^{n}\frac{2^k}{k}-\frac{\log^22}{8}\sum_{k=1}^{n}k2^k\notag\\
&+\frac{\log2}{2}-\frac{1}{4}\sum_{p=1}^{n-1}\bigg\{\sum_{k=1}^{n-p}\bigg(\frac{H_{p+k}-H_k}{(p+k)k}+\frac{1}{k^2(p+k)}\bigg)\bigg\}(p+p^2)2^p\notag\\
&-\frac{2-\log2}{4}\sum_{k=1}^{n}H_k2^k+\frac{\log2}{4}\sum_{k=1}^{n}kH_k2^k+\frac{4\log2-\log^22}{8}\sum_{k=1}^{n}2^k.
\end{align*}
\end{exam}
%***************************************************************************************************Example 18
\begin{exam} For $n=2,3,4,...$ and $t=1$ in (\ref{e:54}) and (\ref{e:55}) we get, respectively
\begin{align*}
&\sum_{k=1}^{\infty}\frac{1}{k\binom{n+k}{k}}=\frac{1}{n}\quad\mbox{}and\quad \sum_{k=1}^{\infty}\frac{1}{\binom{n+k}{k}}=\frac{1}{n-1}
\end{align*}
\end{exam}
%***************************************************************************************************Example 19
\begin{exam} For $n\in\mathbb{N}$  and $t=-1$ in (\ref{e:60})  we get
\begin{align*}
&\sum_{k=1}^{\infty}\frac{(-1)^k}{k\binom{n+k}{k}}=\frac{1}{2}\sum_{k=1}^{n}\{H_k+H_{n-k}-H_n\}2^k+\sum_{k=1}^{n}\frac{2^{k-1}}{k}-2^n\log2.
\end{align*}
\end{exam}
%***************************************************************************************************Example 20
\begin{exam} For $n\in\mathbb{N}$  and $t=-1$ in (\ref{e:61})  we get
\begin{align*}
\sum_{k=1}^{\infty}\frac{(-1)^k}{\binom{n+k}{k}}&=-\frac{1}{4}\sum _{k=1}^n(k+1)\left(H_k+H_{n-k}-H_{n}\right)2^k\notag\\
&+\sum _{k=0}^{n-2} \frac{2^k}{k+2}+(n\log2-2)2^{n-1}+1.
\end{align*}
\end{exam}
\begin{rem}
Using the simple relation $H_k=H_{k-1}+\frac{1}{k}$, and Theorems 2.11-2.19 it is possible to give closed form identities for the sums
$$
\sum_{k=1}^{\infty}\frac{H_kt^k}{k^p\binom{n+k}{k}}
$$
for $p=0$ and $1$, $n\in\mathbb{N}$, and $|t|\leq 1$.
\end{rem}
%***************************************************************************************************Example 21
\begin{exam} For $n=2$ and $t=0$ in (\ref{e:63}), we have by (\ref{e:3}) and (\ref{e:4})
\begin{equation*}
\zeta(3)=\sum_{k=1}^{\infty}\frac{H_{k-1}}{k^2}.
\end{equation*}
\end{exam}
\begin{rem}This identity was discovered by Euler and has a long history; see \cite[p. 252]{6}, and also \cite{8}.
\end{rem}
%***************************************************************************************************Example 22
\begin{exam} For $n=1$ and $t=1/2$ in (\ref{e:62}), we have by (\ref{e:1}) and (\ref{e:6})
\begin{equation*}
\zeta(2)=\frac{4}{3}+\frac{1}{3}\sum_{k=1}^{\infty}\frac{4^k}{k^2(2k+1)\binom{2k}{k}}.
\end{equation*}
\end{exam}
%***************************************************************************************************Example 23
\begin{exam} For $n=1$ and $t=-1/2$ in (\ref{e:62}), we have by (\ref{e:3}) and (\ref{e:4})
\begin{equation*}
\zeta(2)=\frac{1}{3}\sum_{k=1}^{\infty}\frac{4^k}{k^2\binom{2k}{k}}.%DOĞRU
\end{equation*}
\end{exam}
%***************************************************************************************************Example 24
\begin{exam} For $n=2$ and $t=-1/2$ in (\ref{e:62}), we have by (\ref{e:1}) and (\ref{e:6})
\begin{equation*}
\zeta(3)=\frac{1}{7}\sum_{k=1}^{\infty}\frac{H_{k-1}4^k}{k^2\binom{2k}{k}},
\end{equation*}
which is a new series representation for the Ap\`{e}ry constant $\zeta(3)$.
\end{exam}
%***************************************************************************************************Example 25
\begin{exam} For $n=2$ and $t=1/2$ in (\ref{e:62}), we have by (\ref{e:1}) and (\ref{e:6})
\begin{equation*}
\zeta(3)=\frac{8}{7}+\frac{1}{7}\sum_{k=1}^{\infty}\frac{H_{k-1}4^k}{k^2(2k+1)\binom{2k}{k}}.
\end{equation*}
\end{exam}
%***************************************************************************************************Example 26
\begin{exam} For $n=3$ and $t=-1/2$ in (\ref{e:62}), we have by (\ref{e:1}) and (\ref{e:6})
\begin{equation*}
\zeta(4)=\frac{1}{30}\sum_{k=1}^{\infty}\frac{\left(H_{k-1}^2-H_{k-1}^{(2)}\right)4^k}{k^2\binom{2k}{k}}.
\end{equation*}
\end{exam}
%***************************************************************************************************Example 27
\begin{exam} For $n=3$ and $t=0$ in (\ref{e:62}), we have by (\ref{e:3}) and (\ref{e:4})
\begin{equation*}
\zeta(4)=\frac{1}{2}\sum_{k=1}^{\infty}\frac{H_{k-1}^2-H_{k-1}^{(2)}}{k^2}.%DOĞRU
\end{equation*}
\end{exam}
%***************************************************************************************************Example 28
\begin{exam} For $n=3$ and $t=1/2$ in (\ref{e:62}), we have by (\ref{e:1}) and (\ref{e:6})
\begin{equation*}
\zeta(4)=\frac{16}{15}+\frac{1}{30}\sum_{k=1}^{\infty}\frac{\left(H_{k-1}^2-H_{k-1}^{(2)}\right)4^k}{k^2(2k+1)\binom{2k}{k}}.%DOĞRU
\end{equation*}
\end{exam}
%***************************************************************************************************Example 29
\begin{exam} For $n=4$ and $t=0$ in (\ref{e:62}), we have by (\ref{e:3}) and (\ref{e:4})
\begin{equation*}
\zeta(5)=\frac{1}{6}\sum_{k=1}^{\infty}\frac{H_{k-1}^3-3H_{k-1}H_{k-1}^{(2)}+2H_{k-1}^{(3)}}{k^2}.
\end{equation*}
\end{exam}
%***************************************************************************************************Example 30
\begin{exam} For $n=4$ and $t=-1/2$ in (\ref{e:62}), we have by (\ref{e:1}) and (\ref{e:6})
\begin{equation*}
\zeta(5)=\frac{1}{186}\sum_{k=1}^{\infty}\frac{\left(H_{k-1}^3-3H_{k-1}H_{k-1}^{(2)}+2H_{k-1}^{(3)}\right)4^k}{k^2\binom{2k}{k}}.
\end{equation*}
\end{exam}
\begin{rem}
Examples 21, 27 and 29 are not new and they can be deduced from formula (\ref{e:11}), but it seems to us that Examples 22-26, 28 and 30 are new.
\end{rem}
\textbf{\large{Acknowledgments.}} We are grateful to the anonymous referees for their constructive suggestions.
\bibliographystyle{amsplain}

\end{document}